\documentclass[12pt]{amsart}

\usepackage{amssymb,array}
\usepackage{amsmath}
\usepackage{amsthm}
\usepackage{amsbsy}
\usepackage{bm}
\usepackage[english]{babel}
\usepackage{graphicx}
\usepackage{color}
\usepackage{graphicx}
\usepackage{color}
\usepackage{hyperref}
\usepackage[toc,page]{appendix}

\theoremstyle{plain}
\newtheorem{theorem}{Theorem}[section]
\newtheorem*{theorem*}{Theorem}
\newtheorem{lemma}[theorem]{Lemma}

\newtheorem{proposition}[theorem]{Proposition}

\theoremstyle{definition}
\newtheorem{definition}[theorem]{Definition}
\newtheorem{remark}{Remark}

\newcommand{\A}{x}
\newcommand{\B}{y}
\newcommand{\C}{\mathcal{C}}
\newcommand{\Z}{\mathbb{Z}}
\newcommand{\R}{\mathbb{R}}
\newcommand{\g}{\gamma}
\newcommand{\e}{\epsilon}

\renewcommand{\H}{{\mathbb{H}}}
\renewcommand{\t}{\tau}

\newcommand{\bx}{\langle x \rangle}
\newcommand{\by}{\langle y \rangle}
\newcommand{\la}{\langle}
\newcommand{\ra}{\rangle}
\newcommand{\T}{\mathcal{T}}

\theoremstyle{plain} 
\newcommand{\thistheoremname}{}
\newtheorem*{genericthm*}{\thistheoremname}
\newenvironment{style}[1]
  {\renewcommand{\thistheoremname}{#1}%
   \begin{genericthm*}}
  {\end{genericthm*}}

\begin{document}
\title[Equal angle theorem]{Equal angles of intersecting geodesics for every hyperbolic metric}

\author{Arpan Kabiraj}

\address{Department of Mathematics\\ 
		Chennai Mathematical Institute\\
		Chennai 603103, India}

\email{akabiraj@cmi.ac.in\\arpan.into@gmail.com}

\begin{abstract}

We study the geometric properties of the terms of the Goldman bracket between two free homotopy classes of oriented closed curves in a hyperbolic surface. We provide an obstruction for the equality  of two terms in the Goldman bracket, namely  if two terms in the Goldman bracket are equal to each other then for every hyperbolic metric, the angles corresponding to the intersection points are equal to each other. As a consequence, we obtain an alternative proof of a theorem of Chas, i.e. if one of the free homotopy classes contains a simple representative then the geometric intersection number and the number of terms (counted with multiplicity) in the Goldman bracket are the same. 

\end{abstract}
\maketitle
\section{Introduction} Let $F$ be an oriented surface (possibly with boundary). We denote the free homotopy class of an oriented closed curve $x$ in $F$ by $\la x\ra$. Let $\la x\ra$ and $\la y\ra$ be two free homotopy classes of oriented closed curves. Let $x$ and $y$ be two representatives from $\la x\ra$ and $\la y\ra$ respectively, such that they intersect transversally in double points. The Goldman bracket between $\la x\ra$ and $\la y\ra$ is defined as 
$$[\la x\ra,\la y\ra]=\sum_{p\in x\cap y}\e(p)\la x*_p y\ra$$ 
where $x*_py$ denotes the loop product of $x$ and $y$ based at $p$, $x\cap y$ denotes the set of all intersection points between them,  $\e(p)$ denotes the sign of the intersection between $x$ and $y$ at $p$ (in the positive direction).  Abusing notation, we sometimes denote $[\la x\ra,\la y\ra]$ simply by $[x,y]$.

We denote the collection of all free homotopy classes of oriented closed curves on $F$ by $\C$ and  the free module generated by $\C$ by $\Z(\C)$.
We extend the Goldman bracket linearly on $\Z(\C)$. 

In \cite{Gol}, Goldman proved that this bracket is a well defined Lie bracket on $\Z(\C)$. The pair $(\Z(\C),[,])$, is called the \emph{Goldman Lie algebra}.  

\begin{remark}Since the set of free homotopy classes of the sphere is trivial, the corresponding Goldman Lie algebra is trivial. For the torus, the Goldman Lie algebra is well understood (see \cite[Lemma 7.6]{Ch1}). Therefore throughout the paper we only consider surfaces of negative Euler characteristic.  
\end{remark}


Let $\bx$ and $\by$ be two elements in $\C$ and $i(x,y)$ be the \textit{ geometric intersection number} (see Definition \ref{gin}) between $\bx$ and $\by$. From the definition it follows that if $i(x,y)=0$ then $[\bx,\by]=0$. Goldman \cite[Section 5.17]{Gol}, proved that if $\bx$ contains a representative which is a \emph{simple closed curve} (a closed curve without self-intersections) then the converse is also true. He  used convexity properties of length functions on  Teichm{\"u}ller space to prove the following theorem.   
\begin{theorem}\label{Gol}{\em{(Goldman)}}
Let $\bx$ and $\by$  be two free homotopy classes of closed oriented curves in $F$. If $\bx$ contains a simple representative and $[\bx,\by]=0$ then $i(x,y)=0$. 
\end{theorem}

Goldman \cite[5.17, Remark]{Gol} asked whether Theorem \ref{Gol} (which is a topological statement) has a topological proof. In \cite{Ch1}, using HNN extensions and amalgamated products of the fundamental group of a surface,  Chas gave a topological proof of the following generalization of Theorem \ref{Gol}.

\begin{theorem}{\em{(Chas)}}\label{mt}
Let $\bx$ and $\by$ be elements in $\C$. If $x$ contains a simple representative then $i(x,y)$ is same as the number of terms in $[\bx,\by]$ counted with multiplicity. 
\end{theorem}

In this paper, we use tools from hyperbolic geometry to study the terms of the  Goldman bracket between \emph{any} two free homotopy classes of oriented closed curves. In Theorem \ref{stng}, we obtain an obstruction for the equality of two terms of the Goldman bracket. 

\begin{style}{Theorem \ref{stng}}
 If $x$ and $y$ are any two oriented closed geodesics with intersection points $p$ and $q$ such that the terms of the Goldman bracket have 
 the same associated free homotopy classes (i.e., $\la x*_py\ra=\la x*_qy\ra$)  then  the angles between $x$ and $y$ at $p$ and $q$ (defined appropriately) are equal to each other for every hyperbolic metric in $F$.
\end{style} 

The proof of Theorem \ref{stng} is quite elementary. We then use the following non-trivial observation by Kerckhoff (see Lemma \ref{ker}): \emph{If two geodesics intersect at a point $p$ and one of them is simple, then the twist deformation with respect to the simple geodesic changes the angle of intersection at $p$ strictly monotonically}. Combining Theorem \ref{stng} and Lemma \ref{ker}, we  obtain an alternative proof of Theorem \ref{mt}. 

The hyperbolic geometry techniques used in our work are motivated by the topological operations used by Chas in \cite{Ch1} in the following sense: In \cite{Ch1}, using HNN extension (for non-separating curves) and free product with amalgamation (for separating curves), the author wrote the terms of the Goldman bracket as a composition of two types of terms: (a) elements of the fundamental group of the components of $F\setminus \{\text{the simple closed curve}\}$ and (b) elements of the cyclic group generated by the simple closed curve. Then the author used combinatorial group theory to distinguish the conjugacy classes. Now to obtain a geometrical proof of non-conjugacy, it is natural to consider the geodesic representatives of the corresponding terms and study the angles between these geodesic arcs. That is the proof of Theorem \ref{stng}.    

Goldman discovered this bracket while studying the Weil-Petersson symplectic form on Teichm{\"u}ller space. He showed that given two free homotopy classes of closed curves, the Poisson bracket of the corresponding length functions on  the Teichm{\"u}ller space can  be expressed in terms of the lengths of the terms of the Goldman bracket between them (see \cite[Theorem 3.15]{Gol}). 

The use of the work of Kerchkoff and Wolpert on angles (instead of convexity of the length function) to study non-cancellation of Poisson bracket on  Teichm{\"u}ller space was already known (see \cite[page 226]{W}). The novelty in our approach is to present the required results using basics of hyperbolic geometry and to obtain a proof of Theorem \ref{mt} (rather than Theorem \ref{Gol}) in a self-contained manner.  
  
 Later relation between number of terms of the Goldman bracket and geometric intersection number has been studied using tools from both combinatorial group theory (see \cite{Ch},\cite{Ch1}, \cite{CK}) and hyperbolic geometry (see \cite{GC}, \cite{K1}). For a survey about these results see \cite{Chas}. 

\textbf{Idea of the proofs:} We use techniques from hyperbolic geometry to prove our results. Let $X$ be any point in the Teichm{\"u}ller space $\T(F)$ of $F$ (see Section \ref{Prel}). Let $x_1$  and $y_1$ be two oriented closed curves. The free homotopy class of $x_1$ (respectively $y_1$) contains a unique geodesic, which we denote by $x$ (respectively $y$).  

Given any two terms of $[x,y]$,  we consider the lifts of the terms to the upper half plane $\mathbb{H}$. If two terms corresponding to two intersection points are freely homotopic, then the length of their geodesic representatives with respect to $X$ is the same.  

We use the geometry of the product of geodesics in $\H$ and cosine rule of hyperbolic triangles to show that the angle of intersection (in the positive direction of both axes) only depends on the lengths of the curves $x$ and $y$ together with the length of the geodesic representative of the corresponding term of the Goldman bracket. As the lengths of the geodesic representatives of the corresponding terms are the same, the angles of intersection at both points must be the same. This hold for any $X$ in $\T(F)$ (Theorem \ref{stng}), and gives an obstruction for the equality of two terms.

For the proof of Theorem \ref{mt} we consider $x$ to be simple. We use Fenchel-Nielsen twist deformation to construct a new point $Y$ in $\T(F)$ where the angle of intersection at both points are different which leads to a contradiction (Figure \ref{opoaxis}). The main ingredient to show that the angles are different is Lemma \ref{ker}.

\textbf{Organization of the paper:} In Section \ref{Prel} we recall some basic results from hyperbolic geometry and Teichm{\"u}ller space. In Section \ref{closed} we prove that it is enough to consider the problem for closed surfaces. In Section \ref{lifts} and Section \ref{differ} we describe the lifts of two terms of Goldman bracket which are equal and show the obstruction for the equality of two terms in Theorem \ref{stng}. In Section \ref{chas} we give an alternative proof of  Theorem \ref{mt}. In Appendix \ref{appen} we provide  a proof of Lemma \ref{ker}. In Appendix \ref{trans} we discuss a small technical point regarding transverse intersection and double points.

\vspace*{.75cm}
\noindent\textbf{Acknowledgements:} The author would like to thank Siddhartha Gadgil for his encouragement and enlightening conversations. The author would also like to thank Moira Chas for her help and support. The author is supported by the Department of Science \& Technology (DST); INSPIRE faculty.  

\section{preliminaries}\label{Prel}
In this section we recall some basic facts about hyperbolic geometry, hyperbolic surfaces and Teichm{\"u}ller space. References for the results mentioned in this section are \cite{B},  \cite{Pr}, \cite{ker}.

Let $F$ be an oriented surface of negative Euler characteristic, i.e. $F$ be an oriented  surface of genus $g$ with $b$ boundary components and $n$ punctures such that, $2-2g-b-n< 0.$  Using uniformization theorem we can endow $F$ with a hyperbolic metric (possibly with punctures) with geodesic boundary. 

\subsection{Teichm{\"u}ller space \& hyperbolic geometry}
The Teichm{\"u}ller space $\T(F)$ of $F$, is defined as follows. Consider a pair $(X,\phi)$ where $X$ is a finite area surface with a hyperbolic metric and totally geodesic boundary and $\phi:F\rightarrow X$ is a diffeomorphism. We call $(X,\phi)$ a \emph{marked hyperbolic surface} and $\phi$ the \emph{marking} of $X$. We say $(X_1,\phi_1)$ and $(X_2,\phi_2)$ are equivalent if there exists an isometry $I:X_1\rightarrow X_2$ such that $I\circ \phi_1$ is homotopic to $\phi_2$. The Teichm{\"u}ller space  $\T(F)$ of $F$ is the space of all equivalence classes of marked hyperbolic surfaces. Abusing notation we denote the point $(X,\phi)$ in $\T(F)$ by $X$.

 By a \emph{hyperbolic surface $F_X$} we mean the surface $F$ together  with the point $X$ in $\T(F)$. When the choice of $X$ is clear from the context, we denote the hyperbolic surface $F_X$ simply by $F$. 
 
Given any hyperbolic surface $F$, we obtain an identification of the fundamental group $\pi_1(F)$ with a discrete subgroup of $PSL_2(\R)$ (the group of orientation preserving isometries of the upper half plane $\H$) up to conjugacy in $PSL_2(\R)$. We implicitly use this identification throughout the paper. The action of $\pi_1(F)$ on $\H$ is properly discontinuous and fixed point free. Therefore the quotient space is isometric to $F.$ Henceforth --unless otherwise specified-- by an isometry of $\H$, we mean an orientation preserving isometry.

A homotopically non-trivial closed curve in $F$ is called \emph{essential} if it is not homotopic to a puncture. By a \emph{lift} of a closed curve $\g$ to $\H$, we mean the image of a lift $\R\rightarrow \H$ of the map $\g\circ\pi$ where $\pi:\R\rightarrow S^1$ is the usual covering map. 

There are three types of isometries of $\H$, elliptic, parabolic and hyperbolic. An hyperbolic isometry $f$ has exactly two fixed points in the boundary $\partial\H$ of $\H$, one attracting and one repelling. The oriented geodesic from the repelling fixed point to the attracting fixed point is called the \emph{axis} of $f$ and is denoted by $A_f$. The isometry $f$ acts on $A_f$ by translation by a fixed positive number, called the \emph{translation length} of $f$ which we denote by $\t_f.$  

Since $\pi_1(F)$ acts on $\H$ without fixed points, $\pi_1(F)$ does not contain elliptic elements. Essential closed curves in $\pi_1(F)$ correspond to hyperbolic isometries and closed curves homotopic to punctures correspond to parabolic isometries.

\begin{figure}[h]
	\centering
	\includegraphics[trim = 65mm 45mm 50mm 38mm, clip, width=10cm]{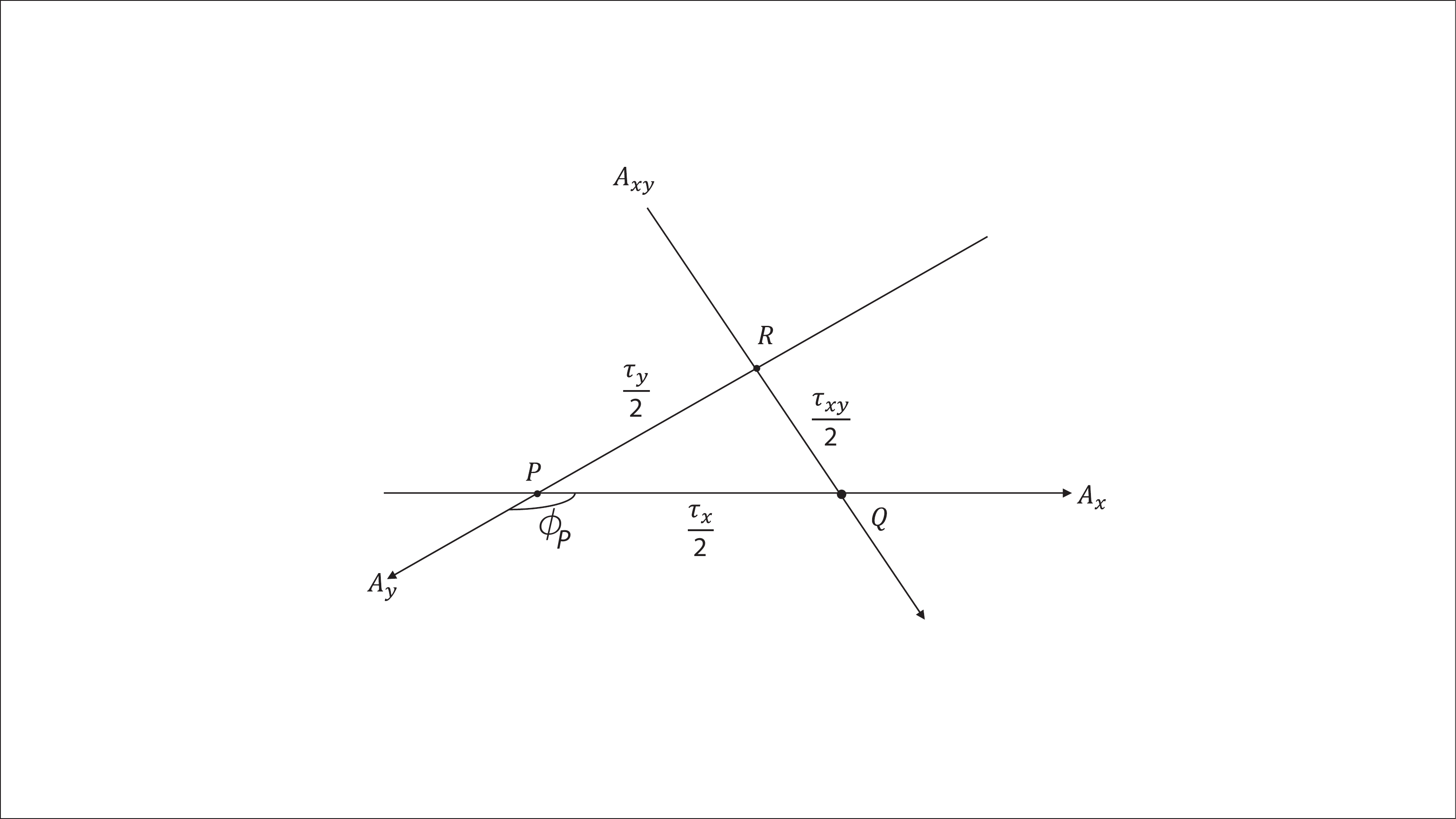}
	\caption{Axis of $xy$.}\label{axis}
\end{figure}

Let $x,y\in \pi_1(F)$ be two hyperbolic elements whose axes intersect at a point $P$ (Figure \ref{axis}). By \cite[Theorem 7.38.6]{B}, $xy$ is also hyperbolic. Let $Q$ be the point in $A_{x}$ at a distance $\tau_{x}/2$ from $P$ in the positive direction of $A_{x}$ and $R$ be the point in $A_{y}$ at a distance $\tau_{y}/2$ from $P$ in the negative direction of $A_{y}.$ Then the unique geodesic joining $R$ and $Q$ with orientation  from $R$ to $Q$ is the axis of $xy$ and the  distance between $Q$ and $R$ is $\tau_{xy}/2.$

Since there is a cannonical  bijective  correspondence between the set of all conjugacy classes in $\pi_1(F)$ and the set of all free homotopy classes of oriented closed curves in $F,$  given an oriented closed curve $x$ in $F$, we can (and will) denote both its free homotopy class and the corresponding conjugacy class in $\pi_1(F)$ by $\bx$.  Every free homotopy class of an essential closed oriented curve contains a unique closed oriented geodesic whose length is same as the translation length of any element of the corresponding conjugacy class.

\begin{remark} Let $\la x\ra,\la y\ra$ be any two free homotopy classes of oriented closed curves in $F$. Given any two points $X_1$ and $X_2$ in $\T(F)$, there is a natural identification between:
\begin{itemize}   
\item[a)] the intersection points between the geodesic representatives of $\la x\ra$ and $\la y\ra$ in $X_1$ and 
\item[b)]the intersection points between the geodesic representatives of $\la x\ra$ and $\la y\ra$ in $X_2$.
\end{itemize}
Throughout the paper we use this identification implicitly.	 
\end{remark}

\subsection{Fenchel-Nielsen twist deformation}
Given a point $X$ in $\T(F)$ and a simple closed geodesic $x$ in $X$, we define the \emph{Fenchel-Nielsen left twist deformation of $X$ at time $s$ along $x$} as follows: Cut the surface along $x$ to get a new (possibly disconnected) surface ${M}$ with geodesic boundary. Form a new hyperbolic surface $X_s$ by gluing the two boundary components of $M$ obtained from $x$ with a left twist of distance $s$, i.e. the images of a  point of $x$ in the two boundaries of ${M}$ are distance $s$ apart in the image of $x$ in $X_s$. Notice that when viewed $x$ as a boundary of ${M}$, the orientation of $F$ provides a unique notion of left and right twists along $x$ (i.e. no orientation of $x$ is required). We call the Fenchel-Nielsen left twist deformation just \emph{left twist deformation}. 

 To consider $X_s$ as a point in $\T(F)$, we have to construct a homotopy class of diffeomorphism from $F$ to $X_s$. Let $N$ be a small annular neighbourhood of $x$ in $X$ and $\bar{N}$ be a small annular neighbourhood of the image of $x$ in $X_s$. Define a diffeomorphism from $N$ to $\bar{N}$ which is homotopic to the left twist deformation at time $s$ relative to $\partial N$. From the definition it follows that we only deform $X$ in a small neighbourhood of $x$. Therefore we extend the deffeomorphism from $N$ to $\bar{N}$ to a diffeomorphism from $X$ to $X_s$ by identity. The composition of this deffeomorphism with the marking of $X$ gives the desired  homotopy class of diffeomorphism from $F$ to $X_s$.
 
 Throughout the paper we fix the anticlockwise orientation of the upper half plane $\H$ and for any $X\in \T(F)$ we identify the universal cover of $X$ with $\H$ preserving this orientation.
 
\begin{figure}[h]
	\centering
	\includegraphics[trim = 65mm 75mm 85mm 50mm, clip, width=8cm]{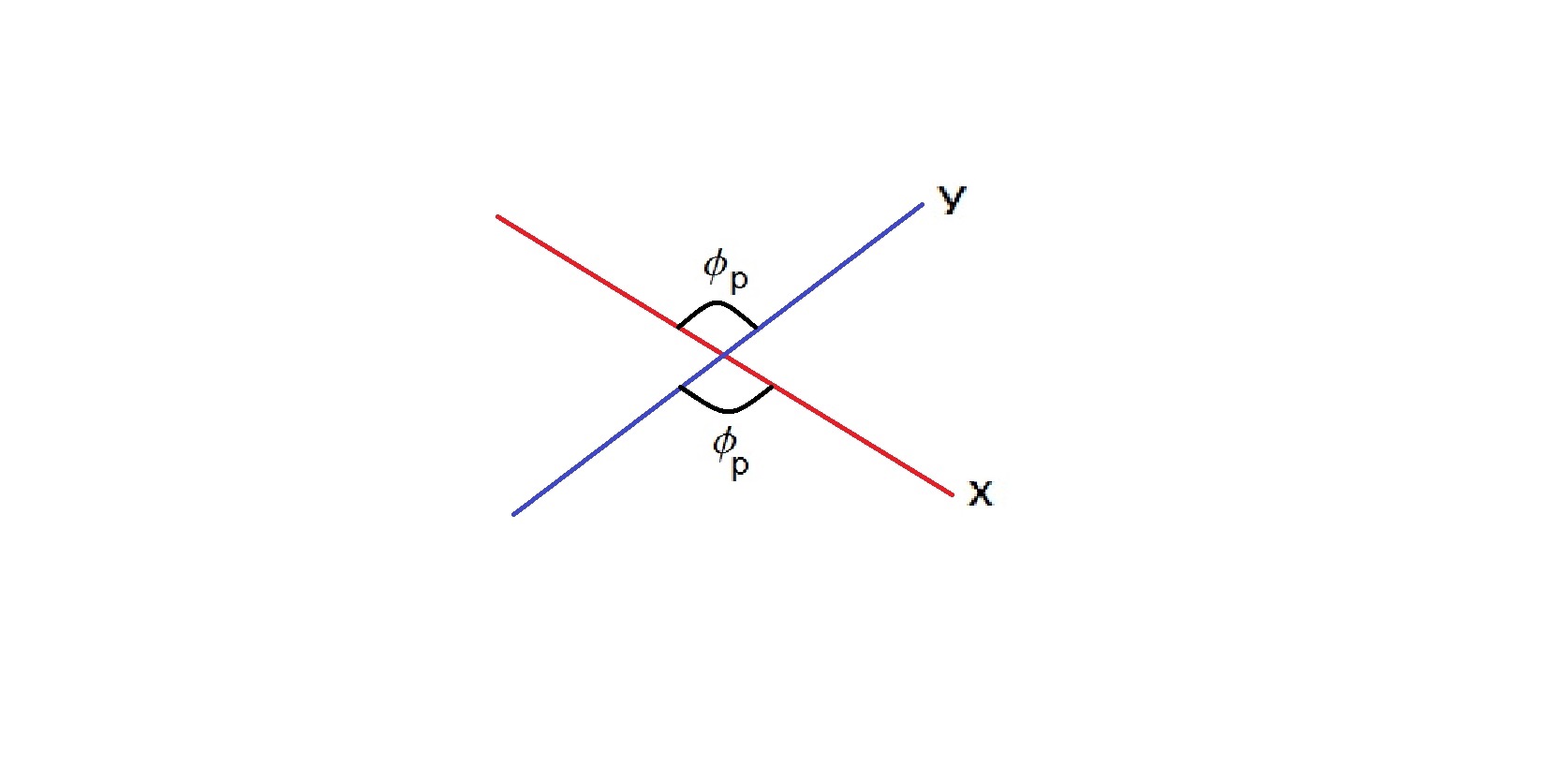}
	\caption{$\phi_p(X)$.}\label{anticlockwise}
\end{figure} 
 
 Let $X$ be a point in $\T(F)$ and $x, y$ be any two intersecting closed geodesics in $X$. Let $p$ be an intersection point between $x$ and $y$ and $\phi_p(X)\in [0,\pi)$ be the angle of intersection between $x$ and $y$ at $p$, where the angle is considered \emph{anticlockwise from $y$ to $x$} (notice that for definition of $\phi_p(X)$, orientation of $x$ and $y$ is not required) see Figure \ref{anticlockwise}. When $X$ is clear from the context, we denote the angle $\phi_p(X)$ simply by $\phi_p$. 
 
 Lemma \ref{ker} stated below,  is crucial to prove Theorem \ref{mt}. If $x$ and $y$ are simple, the lemma follows from \cite[Proposition 3.5]{ker}. Although in \cite[Remark on page 254]{ker} the author mentioned that the same proof works even if $y$ is not simple, for the sake of completeness we give the proof in Appendix \ref{appen}.  

\begin{lemma}\label{ker}
Suppose $X$ is any point in $\T(F)$. Let $x$ be a simple closed geodesic and $y$ be any other closed geodesic in $X$. For every intersection point $p$ between $x$ and $y$, the function $\phi_p(X_s)$ is a strictly decreasing function of $s$.

\end{lemma}

\begin{remark}
We define the angles and ``left" with respect to the anticlockwise orientation of $\H$ and the corresponding orientation of the hyperbolic surface $X\in \T(F)$ (which implies that the angle $\phi_p(X)$ is \emph{strictly decreasing} instead of \emph{strictly monotone}).
\end{remark}

\begin{remark}
The geometric intersection number between a boundary curve or a curve homotopic to puncture and any other closed curve is zero. Also given any hyperbolic surface with punctures, there exists a neighbourhood around every puncture such that every essential closed geodesic is disjoint from these neighbourhoods (i.e. all essential geodesics lie in a compact subset of the surface) (see \cite[Theorem 1.2]{MR}).  Therefore for our discussion there is no difference between a compact surface with boundary and a punctured surface with the same fundamental group. Hence for the rest of the paper we consider only  \emph{compact surfaces}.        
\end{remark}

\begin{remark} In the definition of the Goldman bracket, we would like to choose the  two representatives to be the geodesic representatives of the free homotopy classes. The only problem is, two geodesics not always intersect each other in double points. In  Appendix \ref{trans} we resolve this issue by providing an equivalent definition of Goldman bracket which does not require the representatives to intersect in double points.  
\end{remark}

\section{Surfaces with boundary}\label{closed}
The aim of this section is to prove that for the proof of Theorem \ref{mt}, it is enough to consider surfaces without boundary.

Let $F$ be a hyperbolic surface with geodesic boundary. To each boundary component $\delta$, we attach a hyperbolic surface of genus one with one boundary component of length $l(\delta)$ by an isometry. We call the new surface $\bar{F}$.  Therefore there is a natural inclusion $i: F\rightarrow \bar{F}$. Notice that the gluing preserves the boundary lengths in $F$ and does not change the metric in $F$. Therefore the inclusion $i$ is an isometric inclusion.
\begin{lemma}\label{inclusion}
	The map $i$ induces an injective Lie algebra homomorphism from $\mathcal{L}(F)$ into $\mathcal{L}(\bar{F})$.
\end{lemma}
\begin{proof}
	 We claim that two elements $\A$ and $\B$  are conjugate in $\pi_1(F)$ if and only if $i(\A)$ and $i(\B)$ are conjugate in $\bar{F}$. This follows from the fact that $i$ is an isometric inclusion and every free homotopy class corresponding to the conjugacy class contains a unique geodesic.  
	 
	 Therefore $i$ is an injection from the conjugacy classes in $\pi_1(F)$ to the conjugacy classes in $\pi_1(\bar{F})$. Extending $i$ linearly, we obtain an injective Lie algebra homomorphism from $\mathcal{L}(F)$ into $\mathcal{L}(\bar{F})$.\end{proof}

For the rest of the paper we assume the surfaces to be without boundary unless  otherwise specified.
\section{Lifts of a term in the Goldman bracket}\label{lifts}

Descriptions of the lifts of the terms in the Goldman bracket have been studied in detail in \cite[Section 7]{GC} and \cite[Section 4]{K1}. In this section we provide a self-contained description needed for our results. 

Fix any $X\in\mathcal{T}(F)$ and consider $F$ equipped with the corresponding hyperbolic metric. Consider the term $x*_py$ of the Goldman bracket between two oriented closed geodesics $x$ and $y$ on $F$ corresponding to the intersection point $p$. Choose a lift $P$ of $p$ in $\H$. Without loss of generality we assume that $P\in A_x$. The lift of  $x*_py$ in $\mathbb{H}$, passing through $P$ is a bi-infinite oriented polygonal path $\g$ which is the concatenation of geodesic arcs $\g_i$ with the following description (see Figure \ref{zigzag}): There exists a conjugate $y_0$ of $y$ such that $P=A_x\cap A_{y_0}$. Let $\g_0$ be the geodesic arc of length $\tau_{x}$ on $A_x$ starting from $P$ in the positive direction of $A_x$. There is a conjugate $y_1$ of $y$ such that the endpoint of $\g_0$ is the intersection of $A_x$ and $A_{y_1}$. Let $\g_1$ be the geodesic arc of length $\tau_{y}$ on $A_{y_1}$ starting from the endpoint of $\g_0$ in the positive direction of  $A_{y_1}$. Inductively we define $\g_i$ by the same process for each positive integer $i$. 

\begin{figure}[h]
	\centering
	\includegraphics[trim = 0mm 40mm 10mm 20mm, clip, width=6.5cm]{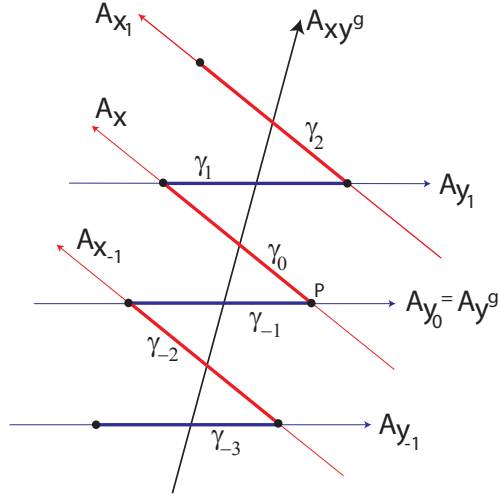}
	\caption{A lift $\g$ of ($x*_py$) to $\H$. It is a concatenation geodesic arcs  $\g_i$. Each $\g_i$ lies on the axis of an element either conjugate to $x$ or conjugate to $y$. The intersection points between $\g_i$ and $\g_{i+1}$ are the lifts of the intersection point $p$.}\label{zigzag}
\end{figure}

Similarly let $\g_{-1}$ be the geodesic arc of length $\tau_{y}$ on $A_{y_0}$  in the negative direction of $A_{y_0}$ ending at $P$. Again there is a conjugate $x_{-1}$ of $x$ such that the beginning of $\g_{-1}$ is the intersection point of $A_{x_{-1}}$ and $A_{y_0}$. Define $\g_{-2}$ to be the geodesic arc of length $\tau_{x}$ on $A_{x_{-1}}$  in the negative direction of $A_{x_{-1}}$ ending at the beginning of $\g_{-1}$.  Inductively we define $\g_i$ by the same process for each negative integer $i$.

In every free homotopy class of curves with endpoints in $\partial\H$ in  $\H$ (fixing endpoints in $\partial\H$) there is a unique hyperbolic geodesic. Suppose $y_{0}=y^g$ for some $g\in \pi_1(F)$. From the description of the product of isometries in Section \ref{Prel} and the symmetry of the Figure \ref{zigzag} around the geodesic $A_{xy^g}$, we have

\begin{proposition}\label{axlemma}
	 The geodesic in the free homotopy class of $x*_py$ is the projection of the axis of  $xy^g$ on $X$.  
	 Furthermore the axis of the geodesic of $xy^g$ intersects each $\g_i$ at their midpoints.
\end{proposition}
We call the geodesic $A_{xy^g}$ to be the \emph{axis of the lift} $\g$.

\section{Proof of Theorem \ref{stng} }\label{differ}

Let $X$ be a point in $\T(F)$ and $x, y$ be two oriented closed geodesics in $X$. Suppose $p$ and $q$ are two intersection points between $x$ and $y$ such that $\la x*_py\ra=\la x*_qy\ra$.  

 For any two oriented geodesics, by angle between them  at any intersection point we mean the angle which is in between the positive direction of both curves, unless otherwise specified. For any point $X\in \T(F)$, we denote $\theta_p(X)$ to be the angle of intersection between $x$ and $y$ at $p$  in $X$.  When $X$ is clear from the context, we denote the angle $\theta_p(X)$ simply by $\theta_p$.
 
 \begin{remark} The angles $\theta_p(X)$ and $\phi_p(X)$ are either congruent or supplementary. 
 \end{remark}
 
 Let $P$ and $Q$ be two lifts of $p$ and $q$ respectively in $\H$, lying in $A_x$. There exist two conjugates $y_1$ and $y_2$ of $y$ such that $P=A_x\cap A_{y_1}$ and $Q=A_x\cap A_{y_1}$. Let $R_1$ (respectively $R_2$) be the point in $A_x$ at a distance $\tau_x/2$ from $P$ (respectively from $Q$) in the positive direction of $A_x$. Similarly let  $S_1$ (respectively $S_2$) be the point in $A_{y_1}$ (respectively $A_{y_2}$) at a distance $\tau_y/2$ from $P$ (respectively from $Q$) in the negative direction of $A_{y_1}$ (respectively $A_{y_2}$). By  definition, the angle between $PR_1$ and $PS_1$ at $P$ is $\pi-\theta_p$ and the angle between $QR_2$ and $QS_2$ at $Q$ is $\pi-\theta_q$ (see Figure \ref{axsign}).
 
  \begin{figure}[h]
  \centering
    \includegraphics[trim = 10mm 30mm 90mm 10mm, clip, width=12cm]{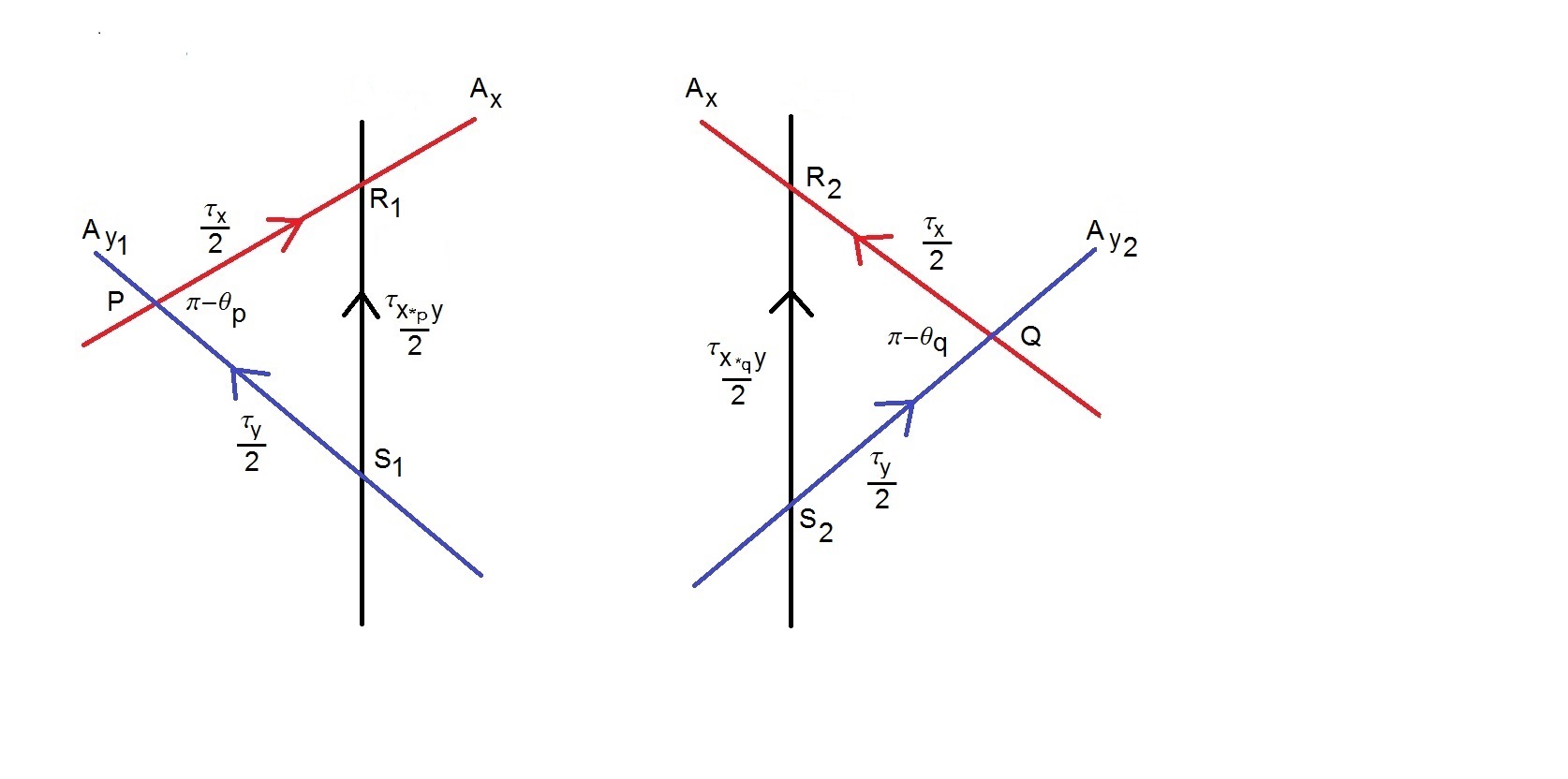}
     \caption{In the above figure we have chosen $\epsilon(p)=\epsilon (P)$ to be $+1$ and $\epsilon (q)=\epsilon (Q)$ to be $-1$. }\label{axsign}
\end{figure}

Consider the triangles $\Delta PR_1S_1$ and $\Delta QR_2S_2$. From the above description we have $$l_X(PR_1)=l_X(QR_2)=\tau_x/2\,\,\,\,\text{and}\,\,\,\, l_X(PS_1)=l_X(QS_2)=\tau_y/2.$$

Also from the Proposition \ref{axlemma} and the assumption $\la x*_py\ra=\la x*_qy\ra$, we have 
 \begin{eqnarray}\label{eqnl}
 \frac{l_X(x*_py)}{2}=l_X(R_1S_1)=l_X(R_2S_2)=\frac{l_X(x*_qy)}{2}.
 \end{eqnarray}
 Therefore by the cosine rule of hyperbolic triangles \cite[\S 7.12]{B}, we have
\begin{eqnarray}\label{eqn} 
 \theta_p(X)=\theta_q(X) . 
 \end{eqnarray}
 
 As $X\in \T(F)$ is arbitrary, the above equation holds for all $X\in\T(F)$. Therefore we have to following theorem.

 \begin{theorem}\label{stng}  Let $X$ be a point in $\T(F)$. If $x$ and $y$ are two oriented closed geodesics in $X$ with intersection points $p$ and $q$ such that $\la x*_py\ra=\la x*_qy\ra$ then  $\theta_p(Y)=\theta_q(Y)$ for all $Y\in \T(F).$
 \end{theorem}

\begin{remark}\label{rmkstng} For the Theorem \ref{stng}, \emph{we do not need $x$ to be simple}.  Clearly Equation \ref{eqn} gives an obstruction for the equality of terms in the Goldman bracket. In the next section we show that if we assume $x$ to be simple in Theorem 
\ref{stng} and $\e(p)\neq \e(q)$ then there exists $X\in \T(F)$ such that  $\theta_p(X)\neq \theta_q(X)$. It would be interesting to see other examples for which the Equation \ref{eqn} fails.
\end{remark} 

\begin{remark} The above obstruction is geometric not topological. In \cite [Problem 13.4]{Ch1}, Chas asked the following question: ``How does one characterize topololgically pairs of intersection points for which the corresponding terms in the Goldman bracket cancel?" Although Theorem \ref{stng} does not solve the problem, the fact that the angles are congruent for \emph{all} metrics makes it \emph{almost} topological. It might give a hint of how to find a characterization in topological terms. Also the duality between the equalities \ref{eqnl} and \ref{eqn} can be used to construct length equivalent curves (see \cite{K2}). For topological properties of length equivalent curves see \cite{L}.  
\end{remark}

\section{Proof of the Theorem \ref{mt}}\label{chas}
We prove the following lemma from which Theorem \ref{mt} follows. 

\begin{lemma}
Let $X$ be any point in $\T(F)$ and  $x,y$ be two oriented closed geodesics on $X$ intersecting each other. Suppose $x$ is a simple geodesic. If $p,q\in x\cap y$ such that $\e(p)=-\e(q)$ then $\la x*_py\ra\neq \la x*_qy\ra$.
\end{lemma}

\begin{proof}

\begin{figure}[h]
  \centering
    \includegraphics[trim = 30mm 60mm 30mm 35mm, clip, width=12cm]{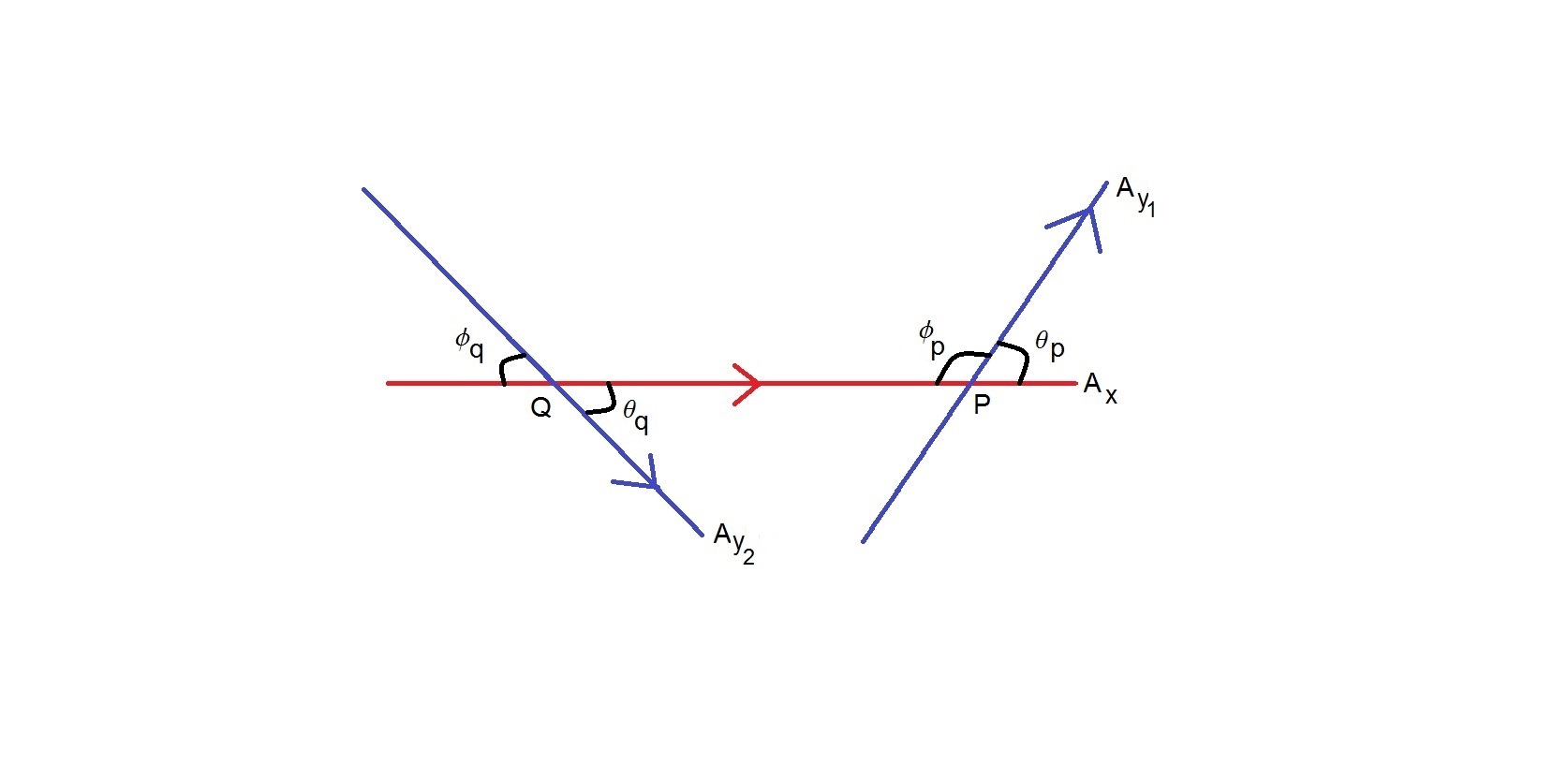}
     \caption{Two lifts $P$ and $Q$ of the points $p$ and $q$ respectively.  }\label{opoaxis}
\end{figure} 

We prove the lemma by contradiction. Suppose $\e(p)=-\e(q)$ and $\la x*_py\ra= \la x*_qy\ra$.

We have two possibilities: 
\begin{itemize}
\item[(1)] $\e(p)=+1$. In this case $\theta_p+\phi_p=\pi$.
\item[(2)] $\e(p)=-1$. In this case $\theta_p=\phi_p$.
\end{itemize}

As proofs for both the cases are identical, without loss of generality we assume that $\e(p)=+1$ and $\e(q)=-1$. Therefore  $\theta_p+\phi_p=\pi$ and $\theta_q=\phi_q$.

As $x$ is simple, let $X_s$ be the point in $\T(F)$ obtained by a left twist deformation after time $s$ from $X$ along $x$. Since $\e(p)=-\e(q)$, 
$$\phi_p(X_s)=\pi-\theta_p(X_s)\,\,\,\,\text{and}\,\,\,\,\phi_q(X_s)=\theta_q(X_s)\,\,\,\, \text{ (see Figure \ref{opoaxis}}).$$ By Lemma \ref{ker}, both the  functions $\phi_p(X_s)$ and $\phi_q(X_s)$ are strictly decreasing. Therefore  $\theta_p(X_s)$ is an strictly increasing function and $\theta_q(X_s)$ is a strictly decreasing function. On the other hand by Theorem \ref{stng}, $\theta_p(X_s)=\theta_q(X_s)$. Thus we arrived at a contradiction as desired.\end{proof}

\appendix\label{appen}
\section{Proof of Lemma \ref{ker}}\label{appen}

\begin{style}{Lemma \ref{ker}}
	Let $x$ be a simple closed geodesic and $y$ be any closed geodesic in $X$. For every intersection point $p$ between $x$ and $y$, the function $\phi_p(X_s)$ is a strictly decreasing function of $s$.
	
\end{style}
\begin{proof}
	We use the unit disc model of the hyperbolic plane. Without loss of generality, assume that the axis of the simple geodesic $x$, $A_x$ is the horizontal diameter. 
	
   Since $x$ is simple, all lifts of $x$ are disjoint. Fix a lift $P$ of $p$ in $A_x$ and choose a lift $\g$ of $y$ passing through $P$. Let $\g_s$ be the image of $\g$ after left twist deformation along $x$ at time $s$ and $\overline{\g_s}$ is the geodesic corresponding to $\g_s$. When we travel from $P$ along $\g_s$, every time we cross a lift of $x $, we have to slide a distance $s$ to the left along that lift. Therefore viewed from $P$, $\g_s$ is an alternative concatenation of geodesic arcs $A_i$ and $B_i$ as shown in Figure \ref{Twist}. 
	
   To prove the lemma it is enough to show that the endpoints of $\overline{\g_s}$ are strictly to the left (when viewed from a $P$) of the endpoints of $\g$. 
   
   Consider the geodesic rays (the dotted lines in the figure) obtained by extending $A_i$ in the forward direction viewed from $P$. We also denote this rays by $A_i$. We claim that for $i$ positive (respectively negative), the endpoint of the ray $A_{i+1}$ (respectively $A_{i-1}$) lies on the left (when viewed from $P$) of the endpoint of the ray $A_i$. We show it for $i$ positive. For $i$ negative the proof is similar. 
   
   	\begin{figure}[h]
   	\centering
   	\includegraphics[trim = 20mm 48mm 30mm 70mm, clip, width=7cm]{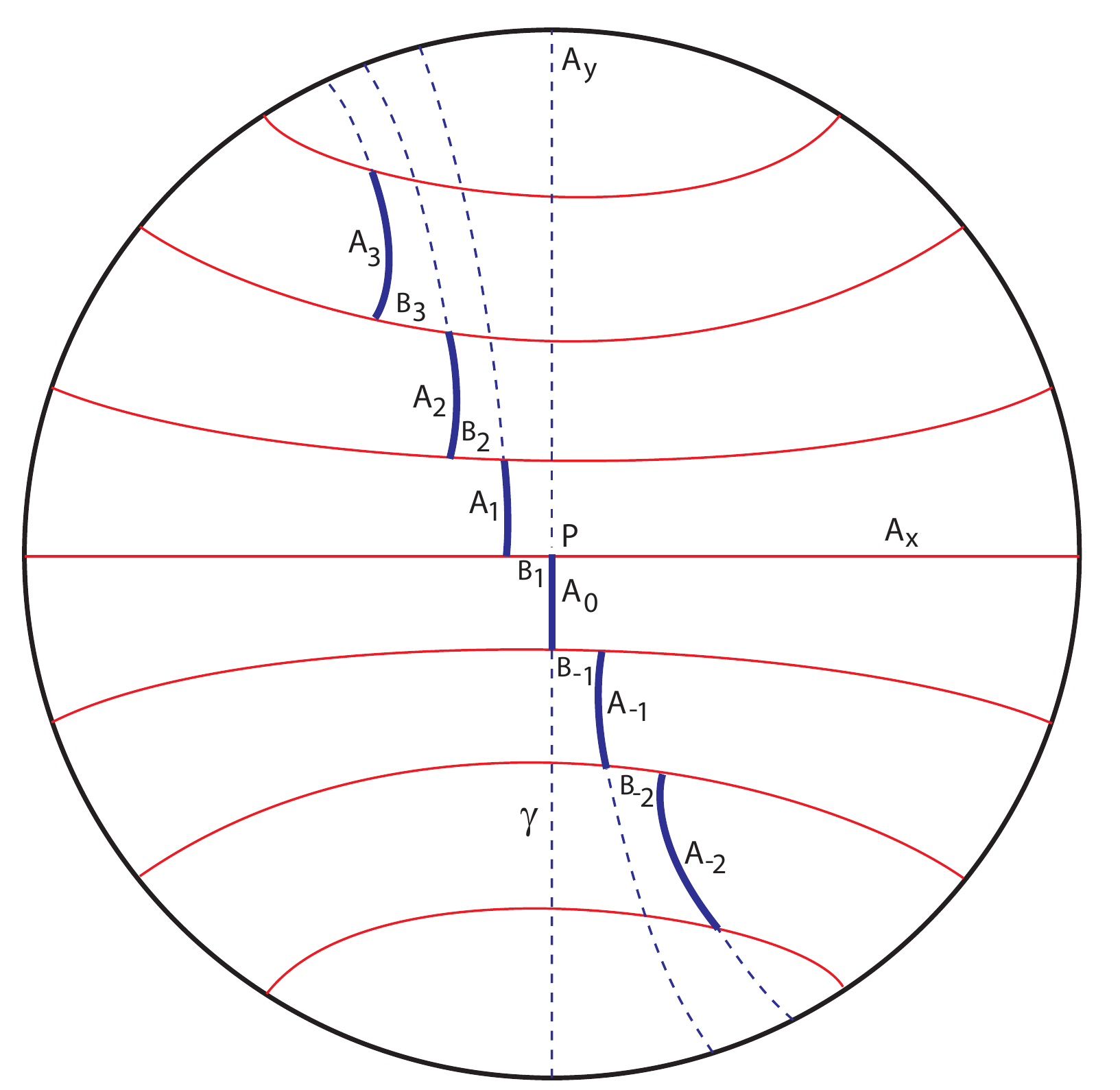}
   	\caption{}\label{Twist}
   \end{figure}
   
   Suppose the endpoint of $A_{i+1}$ is on the right of the endpoint of $A_i$. Then the geodesic rays $A_i$ and $A_{i+1}$ must intersect.  The angle between $B_{i+1}$  with $A_{i-1}$ and $A_i$ are the same (viewed anti-clockwise from $B_i$). Therefore sum of the angles of the triangle formed by the endpoints of $B_{i+1}$ and the intersection point between  $A_i$ and $A_{i+1}$ is at least $\pi$, which is impossible by Gauss-Bonnet theorem. This proves the claim.
   
   By the uniqueness of geodesics in $\H$, when  $i$ goes to $+\infty$ (respectively $-\infty$)  the endpoints of the rays $A_i$ converges to the endpoints of $\overline{\g_s}$. As for $i$ positive (respectively negative), endpoint of each $A_{i+1}$ (respectively $A_{i-1}$) is strictly to the left of the endpoint of $A_{i}$, when viewed from $P$, the endpoints of $\overline{\g_s}$ lies strictly on the left of the endpoints of the geodesic $\g$.\end{proof}

\section{Transverse intersection without double points}\label{trans}

In the definition of the Goldman bracket, we required the two curves to intersect transversally in double points. In this section we show that the condition of intersecting in double points is not necessary.

Let $\A:S^1\rightarrow F$ and $\B:S^1\rightarrow F$ be two smooth curves on $F$, intersecting transversally (not-necessarily in double points). Define the set of all intersection points 
$$I(\A,\B)=\{(t_1,t_2)\in S^1\times S^1: \A(t_1)=\B(t_2)\}.$$
\begin{definition}\label{gin}
	Let $\la x\ra$ and $\la y \ra$ be two free homotopy classes of closed curves. The \emph{geometric intersection number} $i(x,y)$ between $\la x\ra$ and $\la y\ra$ is defined as $$i(x,y)=\underset{x\in \la x\ra,y\in\la y\ra}{\mathrm{inf}}\#I(x,y).$$
\end{definition}	

 Let $x$ and $y$ be two smooth closed oriented curves in $F$. Given any intersection point $(t_1,t_2)\in I(\A,\B)$, let $p=\A(t_1)=\B(t_2)$. Let $\A_*:\pi_1(S^1,t_1)\rightarrow \pi_1(F,p)$ and $\B_*:(S^1,t_2)\rightarrow \pi_1(F,p)$ be the maps induced by $\A$ and $\B$ in the fundamental group of $S^1$ based at $t_1$ and $t_2$ respectively. Let $z_1$ (respectively $z_2$) be the generator of $\pi_1(S^1,t_1)$ (respectively $\pi_1(S^1,t_2)$) with the given orientation of $S^1$. Define the loop product of $\A$ and $\B$ at $(t_1,t_2)$ by $\A_*(z_1)*_{p}\B_*(z_2)$ where $*_{p}$ denotes the product of the fundamental group of $F$ based at $p$. Define the Goldman bracket between $\A$ and $\B$ to be $$[\A,\B]=\sum_{(t_,t_2)\in I(\A,\B)}\e(p)\la\A_*(z_1)*_{p}\B_*(z_2)\big\ra$$ where $\e(p)=\pm1$ depending on whether the orientation of $(\A'(t_1),\B'(t_2))$ agrees with the orientation of $F$ or not.

 It is straightforward to check that this definition agrees with the original definition of the Goldman bracket. The benefit of the above definition is that we can consider only the geodesic representatives in the corresponding free homotopy classes as two geodesics  always intersect each other transversally.

\end{document}